\documentclass[12pt,reqno]{amsart}

\usepackage{fullpage}   
\usepackage{url}

\theoremstyle{plain}
\newtheorem{theorem}{Theorem}[section]
\newtheorem{lemma}[theorem]{Lemma}

\newtheorem{corollary}[theorem]{Corollary}

\newtheorem{problem}[theorem]{Problem}

\newtheorem*{main}{Main Theorem}

\theoremstyle{remark}

\newtheorem*{acknowledgment}{Acknowledgment}

\numberwithin{equation}{section}

\newcommand{\byeqn}[1]{\overset{\eqref{#1}}=}  

\title{A Characterization of Adequate Semigroups By Forbidden Subsemigroups}

\author{Jo\~{a}o Ara\'{u}jo}
\author{Michael Kinyon}
\author{Ant\'{o}nio Malheiro}

\address[Ara\'{u}jo]
{Universidade Aberta \\
R. Escola Polit\'{e}cnica, 147 \\
1269-001 Lisboa, Portugal}

\address[Ara\'{u}jo, Malheiro]
{Centro de \'{A}lgebra \\
Universidade de Lisboa \\
1649-003 Lisboa, Portugal}
\email{\url{mjoao@ptmat.fc.ul.pt}}
\email{\url{malheiro@cii.fc.ul.pt}}

\address[Kinyon]{Department of Mathematics \\
University of Denver \\ 2360 S Gaylord St \\ Denver, Colorado 80208 USA}
\email{\url{mkinyon@du.edu}}

\address[Malheiro]
{Departamento de Matem\'{a}tica, Faculdade de Ci\^{e}ncias e Tecnologia\\
Universidade Nova de Lisboa \\
2829-516 Caparica, Portugal}

\begin{document}

\begin{abstract}
A semigroup is \emph{amiable} if there is exactly one idempotent in each $\mathcal{R}^*$-class and in each $\mathcal{L}^*$-class. A semigroup is \emph{adequate} if it is amiable and if its idempotents commute. We characterize adequate semigroups by showing that they are precisely those amiable semigroups which do not contain isomorphic copies of two particular nonadequate semigroups as subsemigroups.
\end{abstract}

\maketitle

\section{Introduction}

For a semigroup $S$, the usual Green's equivalence relations $\mathcal{L}$ and $\mathcal{R}$ are defined by $x\mathcal{L}y$ if and only if $S^1 x = S^1 y$ and $x\mathcal{R}y$ if and only if $xS^1 = yS^1$ for all $x,y\in S$ where $S^1 = S$ if $S$ is a monoid and otherwise $S^1 = S\cup \{1\}$, that is $S$ with an identity element $1$ adjoined. Naturally linked to these relations are the classes of semigroups defined as follows:
\begin{itemize}
\item A semigroup is \emph{regular} if there is an idempotent in each $\mathcal{L}$-class and in each $\mathcal{R}$-class;
\item A semigroup is \emph{inverse} if there is a unique idempotent in each $\mathcal{L}$-class and in each $\mathcal{R}$-class. A regular semigroup is inverse if and only if its idempotents commute (\cite[Theorem 5.1.1]{Ho}).
\end{itemize}

The starred Green's relation $\mathcal{L}^*$ is defined by $x\mathcal{L}^* y$ if and only if $x\mathcal{L}y$ in some semigroup containing $S$ as a subsemigroup, and a similar definition gives $\mathcal{R}^*$. These are characterized, respectively, by $x\mathcal{L}^* y$ if and only if, for all $a,b\in S^1$, $xa = xb \Leftrightarrow ya = yb$ and by $x\mathcal{R}^* y$ if and only if, for all $a,b\in S^1$, $ax = bx \Leftrightarrow ay = by$.
Naturally linked to these relations are the classes of semigroups defined as follows:

\begin{itemize}
\item A semigroup is \emph{abundant} if there is an idempotent in each $\mathcal{L}^*$-class and in each $\mathcal{R}^*$-class \cite{fountain2}.
\item A semigroup is \emph{adequate} if it is abundant and its idempotents commute \cite{fountain1}.
\item A semigroup is \emph{amiable} if there is a unique idempotent in each $\mathcal{L}^*$-class and in each $\mathcal{R}^*$-class \cite{AK}. Every adequate semigroup is amiable \cite{fountain1}.
\end{itemize}

Since $\mathcal{L}\subseteq \mathcal{L}^*$ and $\mathcal{R}\subseteq \mathcal{R}^*$, abundant semigroups generalize regular semigroups, and amiable (and hence, adequate) semigroups generalize inverse semigroups. Of course, the classes of regular and inverse semigroups are among the most intensively studied classes of semigroups. Many of the fundamental results in these classes have been generalized to abundant and adequate semigroups, for which there is also an extensive literature.

It has been known since Fountain's first paper \cite{fountain1} that the class of adequate semigroups is properly contained in the class of amiable semigroups, because he constructed an infinite amiable, but not adequate, semigroup. Later M. Kambites asked whether these two classes coincide on finite semigroups and a negative answer was provided in \cite{AK}. The aim of this paper is to characterize adequate semigroups inside the class of amiable semigroups. We therefore hope that our main result will provide a useful tool for generalizing results on inverse and adequate semigroups to the setting of amiable semigroups.

We say that a semigroup $S$ \emph{avoids} a semigroup $T$ if $S$ does not contain an isomorphic copy of $T$ as a subsemigroup. The main result of this paper is the following.

\begin{main}
\label{Thm:Main}
Let $S$ be an amiable semigroup. Then $S$ is adequate if and only if $S$ avoids both of the semigroups defined by the presentations
\begin{align*}
\mathcal{F} &= \langle a, b \mid a^2 = a, b^2 = b \rangle  \tag{F} \\
\intertext{and}
\mathcal{M} &= \langle a, b \mid a^2 = a, b^2 = b, aba = bab = ab \rangle\,. \tag{M}
\end{align*}
\end{main}

The semigroup $F$ defined by the presentation $\mathcal{F}$ is  Fountain's original example of an amiable semigroup which is not adequate (\cite{fountain1}, Example 1.4). Except for changes in notation, the example is given as follows. Let
\[
A = \begin{pmatrix} 1 & 0 \\ 1 & 0\end{pmatrix},\qquad
B = \begin{pmatrix} 1 & 1 \\ 0 & 0\end{pmatrix},\qquad
C = \begin{pmatrix} 1 & 1 \\ 1 & 1\end{pmatrix},\qquad
D = \begin{pmatrix} 2 & 0 \\ 0 & 0\end{pmatrix}\,.
\]
Set $F_0 = \{ 2^n A, 2^n B, 2^n C, 2^n D\mid n\geq 0\}$. Then $F_0$ is a semigroup under the usual matrix multiplication. It is easy to see that $A$ and $B$ are the only idempotents of $F_0$. The $\mathcal{L}^*$-classes are $\{ 2^n A, 2^n D\mid n\geq 0\}$, $\{ 2^n B, 2^n C\mid n\geq 0\}$, and the $\mathcal{R}^*$-classes are $\{ 2^n A, 2^n C\mid n\geq 0\}$, $\{ 2^n B, 2^n D\mid n\geq 0\}$. Hence $F_0$ is amiable, but it is not adequate since $AB = C\neq D = BA$.

Note that for $n \geq 0$, $2^n C = C^{n+1} = (AB)^{n+1}$, $2^n D = D^{n+1} = (BA)^{n+1}$, $2^n A = C^n A = (AB)^n A$ and $2^n B = D^n B = (BA)^n B$. Therefore $F_0 = \{ (AB)^n A, (BA)^n B, (AB)^m, (BA)^m \mid n\geq 0, m\geq 1\}$ with no two of the listed elements coinciding. On the other hand,   $F=\{ (ab)^n a, (ba)^n b, (ab)^m, (ba)^m \mid n\geq 0, m\geq 1\}$ with no two of the listed elements coinciding. Therefore Fountain's $F_0$ has $\mathcal{F}$ as a presentation.

The semigroup $M$ defined by the presentation $\mathcal{M}$ is the first known \emph{finite} example of an amiable semigroup which is not adequate \cite{AK}. Setting $c = ab$ and $d = ba$, it is easy to see from the relations that $M = \{a,b,c,d\}$ with the following multiplication table:
\begin{table}[htb]
\[
\begin{tabular}{c|cccc}
  $M$ & $a$ & $b$ & $c$ & $d$ \\ \hline
  $a$ & $a$ & $c$ & $c$ & $c$ \\
  $b$ & $d$ & $b$ & $c$ & $d$ \\
  $c$ & $c$ & $c$ & $c$ & $c$ \\
  $d$ & $d$ & $c$ & $c$ & $c$
\end{tabular}
\]
\caption{The smallest amiable semigroup which is not adequate.}
\end{table}

\noindent The $\mathcal{L}^*$-classes are $\{a,d\}$, $\{b\}$ and $\{c\}$. The $\mathcal{R}^*$-classes are $\{b,d\}$, $\{a\}$ and $\{c\}$. Thus $M$ is amiable but it is evidently not adequate since $ab = c\neq d = ba$.

In fact, the original motivation for this paper was a conjecture offered in \cite{AK}, that every finite amiable semigroup which is not adequate contains an isomorphic copy of $M$. The conjecture was based on a computer search in which it was found that the conjecture holds up to order $37$. The confirmation of this conjecture is a trivial corollary of our Main Theorem.

The preceding discussion has shown that the avoidance condition of the Main Theorem is certainly necessary, since both $F$ and $M$ are amiable but neither is adequate. The next section is devoted to the proof of the sufficiency. In the last section, we pose some problems.

\section{The Proof}
\label{Sec:Finale}

In what follows we make frequent use of the fact that for idempotent elements (more generally, for regular elements) $s,t$ of an abundant semigroup, $s\mathcal{L}^* t$ if and only if $s\mathcal{L}t$ and similarly, $s\mathcal{R}^* t$ if and only if $s\mathcal{R}t$ \cite{fountain2}.

For each $x$ in an amiable semigroup, we denote by $x_{\ell}$ the unique idempotent in the $\mathcal{L}^*$-class of $x$ and we denote by $x_r$ the unique idempotent in the $\mathcal{R}^*$-class of $x$. (In the literature, these are sometimes denoted by $x^*$ and $x^+$, respectively.) We can view amiable semigroups as algebras of type $\langle 2,1,1\rangle$ where the binary operation is the semigroup multiplication and the unary operations are $x\mapsto x_{\ell}$ and $x\mapsto x_r$. Thus we may think of amiable semigroups as forming a quasivariety axiomatized by, for instance, associativity together with these eight quasi-identities
\begin{alignat*}{3}
x_{\ell} x_{\ell} &= x_{\ell} & \qquad  x_r x_r &= x_r \\
x x_{\ell} &= x & \qquad   x_r x &= x \\
xy = xz &\Rightarrow x_{\ell}y = x_{\ell}z & \qquad   yx = zx &\Rightarrow yx_r = zx_r
\end{alignat*}
\begin{align*}
(\ xx = x\ \land\ yy = y\ \land\ x\mathcal{L} y\ ) &\Rightarrow  x = y \\
(\ xx = x\ \land\ yy = y\ \land\ x\mathcal{R} y\ ) &\Rightarrow  x = y\,.
\end{align*}
Here $x\mathcal{L} y$ abbreviates the conjunction $(\ xy = x\ \land\ yx = y\ )$, and similarly $x\mathcal{R} y$ abbreviates $(\ xy = y\ \land\ yx = x\ )$. We will use these quasi-identities in what follows without explicit reference.

\begin{lemma}
\label{Lem:Reduce}
For all $x,y$ in an amiable semigroup,
\begin{equation}
\label{Eqn:Reduce}
(x_{\ell} y)_{\ell} = (xy)_{\ell}\,.
\end{equation}
\end{lemma}

\begin{proof}
The relation $\mathcal{L}^*$ is a right congruence. Since $x\mathcal{L}^*x_{\ell}$, we have $xy\mathcal{L}^*x_{\ell}y$ and so $(x_{\ell} y)_{\ell} = (xy)_{\ell}$.
\end{proof}

\begin{lemma}
\label{Lem:AvoidsM}
Let $S$ be an amiable semigroup and let $a,b\in S$ be noncommuting idempotents. The following are equivalent: (i) $aba = ab$, (ii) $bab = ab$, (iii) $abab = ab$. When these conditions hold, the subsemigroup of $S$ generated by $a$ and $b$ is isomorphic to $M$.
\end{lemma}

\begin{proof}
The equivalence of (i) and (ii) is (\cite{AK}, Lemma 2). If (i) holds, then clearly $abab = abb = ab$ and so (iii) holds. Now assume (iii). Then $aba\cdot aba = ababa = aba$, and so $aba$ is an idempotent. We have $aba\cdot ab = ab$ and $ab\cdot aba = aba$, and so $aba\mathcal{R} ab$. Since $S$ is amiable, $aba = ab$, that is (i) holds. The remaining assertion is (\cite{AK}, Theorem 3).
\end{proof}

We can interpret Lemma \ref{Lem:AvoidsM} in terms of quasi-identities.

\begin{lemma}
\label{Lem:Quasi}
The class of all amiable semigroups which avoid $M$ is a subquasivariety of the quasivariety of all amiable semigroups. It is characterized by the defining quasi-identities of amiable semigroups together with any one of the following:
\begin{align}
(\ xx = x\ \&\ yy = y\ \&\ xyx = xy\  ) \Rightarrow xy = yx\,, \label{Eqn:AvoidsM1} \\
(\ xx = x\ \&\ yy = y\ \&\ xyx = yx\  ) \Rightarrow xy = yx\,, \label{Eqn:AvoidsM2} \\
(\ xx = x\ \&\ yy = y\ \&\ xyxy = xy\  ) \Rightarrow xy = yx\,. \label{Eqn:AvoidsM3}
\end{align}
\end{lemma}

\begin{proof}
If a semigroup $S$ contains a copy of $M$, then \eqref{Eqn:AvoidsM1} is not satisfied since $aba = ca = c = ab$. Conversely, if \eqref{Eqn:AvoidsM1} is not satisfied in $S$, then there exist idempotents $a$ and $b$ with $aba = ab$. By Lemma \ref{Lem:AvoidsM}, $a$ and $b$ generate a copy of $M$. The proofs for the other two cases are similar.
\end{proof}

\begin{lemma}
\label{Lem:xlyl=xyl}
Let $S$ be an amiable semigroup which avoids $M$ and let $c\in S$ be an idempotent. Then for all $x\in S$,
\begin{align}
x(xc)_{\ell} &= xc\,,  \label{Eqn:xxc=xc1} \\
x_{\ell}(xc)_{\ell} &= x_{\ell}c\,. \label{Eqn:xxc=xc2}
\end{align}
\end{lemma}

\begin{proof}
Since $xc = xcc$, we have $(xc)_{\ell} = (xc)_{\ell} c$, and so $c(xc)_{\ell}c = c(xc)_{\ell}$. By \eqref{Eqn:AvoidsM1}, $c(xc)_{\ell} = (xc)_{\ell}c = (xc)_{\ell}$. Thus $xc = xc(xc)_{\ell} = x(xc)_{\ell}$, which establishes \eqref{Eqn:xxc=xc1}, and then \eqref{Eqn:xxc=xc2} follows from \eqref{Eqn:xxc=xc1}.
\end{proof}

\begin{lemma}
\label{Lem:xc=cx}
Let $S$ be an amiable semigroup which avoids $M$, let $c,x\in S$ and assume $c$ is an idempotent. If $cx = xc$, then $cx_{\ell} = x_{\ell}c$.
\end{lemma}

\begin{proof}
Since $xc = cx = cxx_{\ell} = xcx_{\ell}$, we have $x_{\ell}c = x_{\ell}cx_{\ell}$. By \eqref{Eqn:AvoidsM1}, $cx_{\ell} = x_{\ell}c$.
\end{proof}

\begin{lemma}
\label{Lem:abe=ab}
Let $S$ be an amiable semigroup which avoids $M$, let $a,b\in S$ be idempotents and suppose there exists a positive integers $m,n$, with $m>n$ such that $(ab)^{m} = (ab)^n$. Then $(ab)^{n+1} = (ab)^n$.
\end{lemma}

\begin{proof}
Consider the monogenic subsemigroup of $S$ generated by $ab$. Since $(ab)^m=(ab)^n$ it is finite. Hence it has an idempotent element $(ab)^k$, for some $k\in \mathbb{N}$, with $k\leq m$.

Now $b(ab)^k b=b(ab)^k$ which implies by \eqref{Eqn:AvoidsM1} that $b(ab)^k=(ab)^k b=(ab)^k$. Hence $(ab)^{k+1}=a(ab)^k =(ab)^k$.

Since $(ab)^k =(ab)^{k+j}$, for all $j\in \mathbb{N}$, $k\leq m$ and $(ab)^m=(ab)^n$, the result follows.
\end{proof}

\begin{lemma}
\label{Lem:abreduce}
Let $S$ be an amiable semigroup which avoids $M$, let $a,b\in S$ be idempotents and suppose $(ab)^{n+1} = (ab)^n$ for some integer $n > 0$. Then $ab = ba$.
\end{lemma}

\begin{proof}
If $n = 1$, then the desired result follows from \eqref{Eqn:AvoidsM3}. Thus we may assume $n > 1$.
Clearly $(ab)^n$ is an idempotent. Now $a(ab)^n a = (ab)^n a$, and so by \eqref{Eqn:AvoidsM2}, $(ab)^n a = a(ab)^n = (ab)^n$. Similarly, $b(ab)^n b = b(ab)^n$, and so by \eqref{Eqn:AvoidsM1}, $b(ab)^n = (ab)^n b = (ab)^n$. It follows that $(ab)^n$ is in center (and in fact, is an absorbing element) of the subsemigroup of $S$ generated by $a$ and $b$. Since $(ab)^{n-1}a\cdot b = (ab)^{n-1}a\cdot b(ab)^n$, we have
\[
[(ab)^{n-1}a]_{\ell}\cdot b = [(ab)^{n-1}a]_{\ell}\cdot b(ab)^n = [(ab)^{n-1}a]_{\ell}\cdot (ab)^n\,.
\]
Since $(ab)^n$ commutes with $(ab)^{n-1}a$, it also commutes with $[(ab)^{n-1}a]_{\ell}$ by Lemma \ref{Lem:xc=cx}. Thus
\begin{align*}
[(ab)^{n-1}a]_{\ell}\cdot b &= (ab)^n \cdot [(ab)^{n-1}a]_{\ell} = ab\cdot (ab)^{n-1} [(ab)^{n-1}a]_{\ell} \\ &\byeqn{Eqn:xxc=xc1} ab\cdot (ab)^{n-1} a = (ab)^n a = (ab)^n\,.
\end{align*}
Hence $b\cdot [(ab)^{n-1}a]_{\ell}\cdot b = b(ab)^n = (ab)^n = [(ab)^{n-1}a]_{\ell}\cdot b$. Applying \eqref{Eqn:AvoidsM2}, we have
\begin{equation}
\label{Eqn:bcomm}
[(ab)^{n-1}a]_{\ell}\cdot b = b\cdot [(ab)^{n-1}a]_{\ell}\,.
\end{equation}

Next, we compute
\begin{align*}
(ab)^n &= (ab)^{n-1}ab  \\
&\byeqn{Eqn:xxc=xc1} (ab)^{n-1} [(ab)^{n-1}a]_{\ell} \cdot b \\
&\byeqn{Eqn:bcomm} (ab)^{n-1} b\cdot [(ab)^{n-1}a]_{\ell} \\
&= (ab)^{n-1} [(ab)^{n-1}a]_{\ell} \\
&\byeqn{Eqn:xxc=xc1} (ab)^{n-1}a\,,
\end{align*}
which establishes
\[
(ab)^n = (ab)^{n-1} a\,.
\]
Now $(ab)^{n-1} (ab)^n = (ab)^n = (ab)^{n-1} a$, and so $[(ab)^{n-1}]_{\ell} (ab)^n = [(ab)^{n-1}]_{\ell}\cdot a$. Since $(ab)^n$ commutes with $(ab)^{n-1}$, it also commutes with $[(ab)^{n-1}]_{\ell}$ by Lemma \ref{Lem:xc=cx}. Thus \[
[(ab)^{n-1}]_{\ell}\cdot a = (ab)^n [(ab)^{n-1}]_{\ell} = ab\cdot (ab)^{n-1} [(ab)^{n-1}]_{\ell} = ab\cdot (ab)^{n-1} = (ab)^n\,.
\]
Now $a\cdot [(ab)^{n-1}]_{\ell}\cdot a = a(ab)^n = (ab)^n = [(ab)^{n-1}]_{\ell}\cdot a$. By \eqref{Eqn:AvoidsM2}, $a\cdot [(ab)^{n-1}]_{\ell} = [(ab)^{n-1}]_{\ell}\cdot a$. Therefore
\begin{equation}
\label{Eqn:Woohoo}
(ab)^n = a\cdot [(ab)^{n-1}]_{\ell}\,.
\end{equation}

If $n = 2$, then \eqref{Eqn:Woohoo} is $(ab)^2 = a\cdot [ab]_{\ell} = ab$ by \eqref{Eqn:xxc=xc1}. If $n > 2$, then we multiply both sides of \eqref{Eqn:Woohoo} by $(ab)^{n-2}$ on the left. Since $(ab)^{n-2}(ab)^n = (ab)^n$, we get
\[
(ab)^n = (ab)^{n-2}a\cdot [(ab)^{n-2}a\cdot b]_{\ell} = (ab)^{n-2}a\cdot b = (ab)^{n-1}\,,
\]
using \eqref{Eqn:xxc=xc1}. Therefore we have shown that the assumption that $(ab)^{n+1} = (ab)^n$ implies $(ab)^n = (ab)^{n-1}$. Continuing, we eventually reduce this to $(ab)^2 = ab$. By \eqref{Eqn:AvoidsM3}, $ab = ba$.
\end{proof}

\begin{corollary}
\label{Cor:abe}
Let $S$ be an amiable semigroup which avoids $M$, let $a,b\in S$ be idempotents and suppose there exist positive integers $m,n$, with $m>n$ such that $(ab)^{m} = (ab)^n$. Then $ab = ba$.
\end{corollary}

\begin{proof}
This follows from Lemmas \ref{Lem:abe=ab} and \ref{Lem:abreduce}.
\end{proof}

Now let $S$ denote an amiable semigroup which is not adequate and which avoids $M$. We fix noncommuting idempotents $a,b\in S$ and let $H$ denote the subsemigroup  generated by $a$ and $b$. The elements of $H$ are
\begin{equation}
\label{Eqn:H}
H = \{ (ab)^m, (ba)^m, (ab)^n a, (ba)^n b \mid m\geq 1, n \geq 0 \}\,.
\end{equation}
(Note that since $S$ is not necessarily a monoid, we interpret $(ab)^0 a$ to be equal to $a$ and similarly $(ba)^0 b = b$.) Our goal is to show that $H$ is an isomorphic copy of $F$. Comparing the elements of $H$ with those of $F$, we see that it is sufficient to show the elements listed in \eqref{Eqn:H} are all distinct.

\begin{lemma}
\label{Lem:Distinct}
The elements of $H$ listed in \eqref{Eqn:H} are all distinct.
\end{lemma}

\begin{proof}
We show that each possible case of two elements of $H$ coinciding will lead to a contradiction. Because $a$ and $b$ can be interchanged, half of the cases follow from the rest by symmetry. We sometimes use this observation implicitly in the arguments that follow when we refer to already proven cases.

\emph{Case 1}: If $(ab)^m = (ab)^n$ for some $m > n > 0$, then by Corollary \ref{Cor:abe}, $ab = ba$, a contradiction.

\emph{Case 2}: If $(ab)^m = (ab)^n a$ for some $m>0$, $n\geq 0$, then $(ba)^{m+1} = b(ab)^m a = b(ba)^na\cdot a = (ba)^{n+1}$ which, by Case 1, leads to a contradiction if $m\neq n$. Also $(ab)^{n+1} = (ab)^na\cdot b = (ab)^m b = (ab)^m$ which yields a contradiction by Case 1 if $m\neq n + 1$.

\emph{Case 3}: If $(ab)^m = (ba)^n$ for some $m, n > 0$, then $(ab)^m = a(ab)^m = a(ba)^n = (ab)^n a$, which contradicts Case 2.

\emph{Case 4}: If $(ab)^m = (ba)^n b$ for some $m > 0$, $n \geq 0$, then $(ab)^m = a(ab)^m = a(ba)^n b = (ab)^{n+1}$, which contradicts Case 1 if $m\neq n + 1$. Also $(ba)^{m+1} = b(ab)^m a = b(ba)^n ba = (ba)^{n+1}$, which contradicts Case 1 if $m \neq n$.

\emph{Case 5}: If $(ab)^m a = (ba)^n b$ for some $m,n\geq 0$, then $(ab)^{m+1} = (ab)^ma\cdot b = (ba)^nb\cdot b = (ba)^n b$, which contradicts Case 4.

\emph{Case 6}: If $(ab)^m a = (ab)^n a$ for some $m > n\geq 0$, then $(ab)^{m+1} = (ab)^{n+1}$ which contradicts Case 1.

By the symmetry in $a$ and $b$, this exhausts all possible cases of elements of $H$ coinciding. The proof is complete.
\end{proof}

By Lemma \ref{Lem:Distinct}, the semigroup $F$ defined by the presentation $\langle a, b \mid a^2 = a, b^2 = b \rangle$ is the subsemigroup $H$ generated by $a$ and $b$. This completes the proof of the Main Theorem.

\section{Open Problems}

A semigroup is \emph{left} abundant if each $\mathcal{R}^*$-class contains an idempotent, \emph{left} amiable if each $\mathcal{R}^*$-class contains exactly one idempotent and \emph{left} adequate if it is left abundant and the idempotents commute. There are more left amiable semigroups which are not left adequate than just $F$ and $M$. For instance, every right regular band (that is, every idempotent semigroup in which $\mathcal{L}$ is the equality relation) which is not a semilattice is left amiable but not left adequate.

\begin{problem}
Extend the Main Theorem to characterize left amiable semigroups which are not left adequate.
\end{problem}

In \cite{AK} we suggested the problem of characterizing the free objects in the quasivariety of amiable semigroups. Perhaps the following would be more tractable.

\begin{problem}
Determine the free objects in the quasivariety of amiable semigroups which avoid $M$.
\end{problem}

By our Main Theorem, if there is a nonadequate free object in this quasivariety, then it would contain a copy of $F$.

\begin{problem}
Determine if the class of amiable semigroups which avoid $F$ forms a quasivariety, and if so, find explicit characterizing identities.
\end{problem}

The ultimate goal regarding amiable semigroups is the following.

\begin{problem}
Determine to what extent the structure theory of adequate semigroups can be extended to amiable semigroups.
\end{problem}

Finally, the ``tilde'' Green's relation $\widetilde{\mathcal{L}}$ on a semigroup $S$ is defined by $a \widetilde{\mathcal{L}} b$ if and only if, for each idempotent $e\in S$, $ae = a$ if and only if $be = b$. The relation $\widetilde{\mathcal{R}}$ is defined dually. We have $\mathcal{L}\subseteq \mathcal{L}^*\subseteq \widetilde{\mathcal{L}}$ and similarly for the dual relations. A semigroup is \emph{semiabundant} if there is an idempotent in each $\widetilde{\mathcal{L}}$-class and in each $\widetilde{\mathcal{R}}$-class, and a semiabundant semigroup is \emph{semiadequate} if its idempotents commute \cite{FGG}. A semigroup is \emph{semiamiable} if there is a unique idempotent in each $\widetilde{\mathcal{L}}$-class and in each $\widetilde{\mathcal{R}}$-class. Every semiadequate semigroup is semiamiable. There are one-sided versions of all of these notions as well. It is natural to suggest the following.

\begin{problem}
Extend the Main Theorem to characterize (left) semiadequate semigroups among (left) semiamiable semigroups.
\end{problem}

\begin{acknowledgment}
We are pleased to acknowledge the assistance of the automated deduction tool \textsc{Prover9} developed by McCune \cite{McCune}.

The first and third author were partially supported by FCT and FEDER, Project POCTI-ISFL-1-143 of Centro de Algebra da Universidade de Lisboa, and by FCT and PIDDAC through the project PTDC/MAT/69514/2006.
\end{acknowledgment}

\end{document}